\documentclass[
final
]{dmtcs-episciences}

\usepackage[english]{babel}
\usepackage[utf8]{inputenc}
\usepackage{amsmath, amssymb,amsthm}
\usepackage{hyperref}
\usepackage{tikz}
\usepackage{ifthen}
\usepackage[round]{natbib}
\usepackage{subfigure}
\usepackage{mathtools}
\usepackage[shortlabels]{enumitem}

\DeclarePairedDelimiter{\ceil}{\lceil}{\rceil}
\DeclarePairedDelimiter{\floor}{\lfloor}{\rfloor}
\DeclareMathOperator{\Extremal}{ex}
\newtheorem{theorem}{Theorem}
\newtheorem{lemma}{Lemma}
\newtheorem{fact}{Fact}

\author{Ervin Gy\H{o}ri \affiliationmark{1}
\and Runze Wang \affiliationmark{2} 
\and Spencer Woolfson \affiliationmark{3}}

\title{Extremal problems of double stars}

\affiliation{
Alfr\'ed R\'enyi Institute of Mathematics, Budapest, Hungary \\ 
Swarthmore College, Swarthmore, PA, USA \\ 
Hamilton College, Clinton, NY, USA 
}

\received{2021-09-18}

\revised{2022-08-16}

\accepted{2022-09-01}

\publicationdetails{24}{2022}{2}{4}{8499}

\begin{document}
\maketitle
\begin{abstract}
In a generalized Tur\'an problem, two graphs $H$ and $F$ are given and the question is the maximum
number of copies of $H$ in an $F$-free graph of order $n$.
In this paper, we study the number of double stars $S_{k,l}$ in triangle-free graphs. We also study an  opposite version of this question: what is the maximum number of edges and triangles in graphs with double star type restrictions, which leads us to study two questions related to the extremal number of triangles or edges in graphs with degree-sum constraints over adjacent or non-adjacent vertices.

\end{abstract}
\section{Introduction}
In the last decade, generalizations of the extremal function $\Extremal(n,H)$ received more and more attention.  For graphs $G$ and $H$, let $n(H,G)$ denote the number of subgraphs of $G$ isomorphic to~$H$ (referred to as \emph{copies} of $H$). In particular, let $t_3(G) = n(K_3,G)$ denote the number of copies of triangles in $G$.
The first result of this type is due to ~\cite{zykov}  (and also independently by ~\cite{erdos}), who determined $\Extremal(n,K_s,K_t)$ exactly for all $s$ and $t$. Then Erd\H os raised the longstanding conjecture $\Extremal(n,C_5,C_3) = (\frac{n}{5})^5$ (where the lower bound is obtained from the uniformly blown up $C_5$). This conjecture was finally verified quarter of a century later by ~\cite{HHKNR} and independently by ~\cite{Grzesik}. On the other hand, the extremal function $\Extremal(n,C_3,C_5)$ was considered by ~\cite{BGy}. Their results were subsequently improved by ~\cite{AS},~\cite{C5C3}, and~\cite{C5C3v2}, but the problem of determining the correct asymptotics is still open. The problem of maximizing $P_\ell$ copies in a $P_k$-free graph was investigated in~\cite{plpk}.

We begin by considering the following question of this type: among all triangle-free graphs with $n$ vertices, what is the maximum number of ``double stars" $S_{k,l}$ where $S_{k,l}$ denotes a star with degree $k$ and a star with degree $l$ connected by an edge. For example $S_{1,1}$ would be $P_4$, the path with 4 vertices.

Then we study a kind of ``converse" of this double star problem: suppose our forbidden graph is $S_{k,l}$ and what we want to maximize is the number of triangles in the graph, how will the graph look like? We know when our forbidden graph is all double stars $S_{k,l}$ where $k+l = d$, we essentially ask our graph $G$ to satisfy $d_G(x) + d_G(y) < d+2$ for all $xy \in E(G)$.

This leads us to two other related problems. The first problem is to find the \emph{minimum} number of triangles in a graph such that $d_G(x) + d_G(y) \geq n + k$ for all  $xy \in E(G)$ and a constant $k$, provided the graph has no isolated vertex. We know that a complete bipartite graph with $n$ vertices will satisfy $d_G(x) + d_G(y) = n$ for all $xy \in E(G)$, and this is why we need the $+k$ here. Another related problem is to find the minimum number of triangles in a graph such that $d_G(x) + d_G(y) \geq n + 1$ for all $xy \notin E(G), x \neq y$. This condition means that the complement of the graph does not contain any ``big" double stars. This condition relates to Ore's theorem, which claims any graph such that $d_G(x) + d_G(y) \geq n$ for all $xy \notin E(G), x \neq y$ is Hamiltonian.

\section{Double Stars}
The \emph{star} $S_{1,a}$ is defined to be $K_{1,a}$, the complete bipartite graph with parts of size 1 and $a$. The vertex belonging to the part of size 1 is called the \emph{central vertex}, and the \emph{degree} of the star is defined to be $a$. The double star $S_{a,b}$ is the graph formed by connecting the central vertices of stars $S_{a}$ and $S_{b}$ with an edge. For example, $S_{1,1}$ is a path with 4 vertices. We will call the edge connecting the two centers the \emph{central edge} of the double star $S_{a,b}$. We are going to show that if $G$ is triangle-free, then a (not necessarily balanced) complete bipartite graph will maximize the number of double stars $S_{a,b}$ in $G$.

For a triangle-free graph $G$, we can compute the number of double stars $S_{a,b}$ ($a \neq b$) whose central edge is a given edge $uv \in E(G)$. There are two ways to have a double star with central edge $uv$, assuming that $a \neq b$: either $u$ is the central vertex of $S_a$ or $v$ is the central vertex of $S_a$. If $u$ is the central vertex of $S_a$, we may simply choose the $a$ vertices forming a star $S_a$ at $u$ from neighbors of $u$ other than $v$ and choose the $b$ vertices forming a $S_b$ at $v$ from neighbors of $v$ other than $u$. Since $G$ is triangle free, $u$ and $v$ will not share a common neighbor. As a result, we can see that there are 
\begin{math}
    \binom{d_G(u) -1}{a}\binom{d_G(v)-1}{b} + \binom{d_G( u)-1}{b} \binom{d_G(v)-1}{a}
\end{math}
double stars on edge $uv$. Therefore,

\begin{equation}
n(S_{a,b},G) = \sum_{uv \in E(G)} \binom{d_G(u) -1}{a}\binom{d_G(v)-1}{b} + \binom{d_G(u)-1}{b} \binom{d_G( v)-1}{a}.
\end{equation}

First, we are going to show that a triangle-free graph with  maximum degree $\Delta < \frac{n}{2}$ contains strictly fewer copies of $S_{a,b}$ than the balanced complete bipartite graph assuming that $\floor{\frac{n}{2}} \geq \max\{a,b\}+1$. The following fact will be frequently used in subsequent discussion:

\begin{lemma}\label{max-edge-triangle-free}
A triangle-free graph $G$ with $n$ vertices and maximum degree $\Delta$ contains at most $\Delta(n-\Delta)$ edges, and equality holds if and only if $\Delta \geq \frac{n}{2}$ and $G$ is the complete bipartite graph $K_{\Delta,n-\Delta}$.

In particular, the balanced complete bipartite graph $K_{\ceil{n/2},\floor{n/2}}$ is the triangle-free graph with most edges.
\end{lemma}

\begin{proof}
 We know that if $xy \in E(G)$, then $d_G(x) + d_G(y) \leq n$ since otherwise $x$ and $y$ must share a common neighbor $z$ and $x,y,z$ will form a triangle. Now if $v$ is a vertex of maximum degree in $G$, then its $\Delta$ neighbors must have degree at most $n - \Delta$. For the $n-\Delta$ vertices that are not neighbors of $v$, they can have degree at most $\Delta$ since $v$ is the vertex of maximum degree. Therefore, the number of edges is bounded above by
 \begin{equation}
\frac{1}{2} \left[ \Delta(n-\Delta)+(n-\Delta)\Delta \right] = \Delta(n-\Delta).
 \end{equation}
 
 To achieve equality, if $\Delta < \frac{n}{2}$, then $\frac{n}{2} < n-\Delta$, and the maximum number of edges is $\frac{\Delta n}{2} < \Delta(n-\Delta)$. When $\Delta \geq \frac{n}{2}$, we see the $\Delta$ neighbors of the vertex of maximum degree, $v$, must have degree $n-\Delta$ to achieve equality. Since there cannot be any edges between the neighbors of $v$ as this will introduce triangles, we see the $\Delta$ neighbors of $v$ must be connected to the $n-\Delta$ vertices that are not neighbors of $v$. Thus, we will arrive at the complete bipartite graph $K_{\Delta,(n-\Delta)}$. It is now easy to verify that $K_{\Delta,(n-\Delta)}$ is triangle free, has maximum degree $\Delta$ when $d \geq \frac{n}{2}$, and contains $\Delta(n-\Delta)$ edges.

 When $n$ is fixed, since $\Delta$ is an integer, $\Delta(n-\Delta)$ is maximized at $\Delta = \floor{n/2}$ and $\Delta = \ceil{n/2}$. Therefore, we see that the balanced complete bipartite graph is the triangle-free graph with most edges.
\end{proof}

\begin{theorem}
If $G$ is a triangle-free graph with $n$ vertices with maximum degree $\Delta < \frac{n}{2}$, and $\floor{\frac{n}{2}} \geq \max\{a,b\}+1$, then $n(S_{a,b},K_{\floor{n/2}, \ceil{n/2}}) > n(S_{a,b},G)$.
\end{theorem}

\begin{proof}
Suppose $G$ is a triangle-free graph with $n$ vertices and with maximum degree $\Delta < \frac{n}{2}$. We know for each edge $uv \in G$, the number of double stars $S_{a,b}$ with central edge $uv$ is 
\begin{equation}
    \begin{aligned}
            &\binom{d(u) - 1}{a} \binom{d(v) - 1}{b} + \binom{d(u) - 1}{b} \binom{d(v)- 1}{a} \\
    &< \binom{\lfloor n/2 \rfloor - 1}{a}\binom{\lceil n/2 \rceil - 1}{b} + \binom{\lfloor n/2 \rfloor - 1}{b}\binom{\lceil n/2 \rceil - 1}{a}
    \end{aligned}
\end{equation}
since the binomial coefficients $\binom{x}{a}$ are strictly increasing when $x \geq a$ and we have $\floor{n/2} -1 \geq \max \{a,b\}$, which just says for any edge $uv \in E(G)$, the number of double stars with central edge $uv$ is less than the number of double stars with a given edge in $ K_{\floor{n/2}, \ceil{n/2}}$ as central edge. Furthermore, we know that the balanced complete bipartite graph $ K_{\floor{n/2}, \ceil{n/2}}$ maximizes the number of edges for triangle-free graphs with $n$ vertices, so we conclude that $n(S_{a,b},G) < n(S_{a,b}, K_{\floor{n/2}, \ceil{n/2}})$.
\end{proof}

Next, we are going to show even if the graph $G$ has large maximum degree, we can always find a complete bipartite graph $K_{\Delta',n-\Delta'}$ such that $n(S_{a,b},K_{\Delta',n-\Delta'}) \geq n(S_{a,b},G)$. This time, the complete bipartite graph does not have to be balanced:

\begin{theorem}
Suppose $G$ is a triangle-free graph with $n$ vertices and with maximum degree $\Delta \geq \frac{n}{2}$. There exists a $\Delta'$ such that $\frac{n}{2} \leq \Delta' \leq \Delta$ and $n(S_{a,b},K_{\Delta',n-\Delta'}) \geq n(S_{a,b},G)$. Furthermore, if $\floor{\frac{n}{2}} \geq \max\{a,b\}+1$ and $G$ is not a complete bipartite graph, the inequality is strict.
\end{theorem}

\begin{proof}
Let
\begin{equation}
f(x,y) = \binom{x-1}{a}\binom{y-1}{b} + \binom{x-1}{b}\binom{y-1}{a}
\end{equation}
denote the number of double stars on a given central edge $uv$ such that $d_G(u) = x$ and $d_G(v) = y$. If $u+v < n$ and $\Delta \geq \frac{n}{2}$, then either $u$ or $v$ must be strictly less than $\Delta$. Without loss of generality, we may assume that $u < \Delta$. It follows that $u+1 \leq \Delta$ and $f(u,v) \leq f(u+1,v)$. As a result, to find the maximum of $f(u,v)$ under the constraints $u,v \leq \Delta, u+v \leq n$, we only need to consider $u,v$ such that $u+v = n$ and $u,v \leq \Delta$. In other words, we have the following identity:
\begin{equation}
\max_{u,v \leq \Delta,u+v \leq n}f(u,v) = \max_{u,v \leq \Delta, u+v=n} f(u,v).    
\end{equation}
If $u+v = n$, without loss of generality, we may assume that $u \geq \frac{n}{2}$. It follows that
\begin{equation}
\max_{u,v \leq \Delta,u+v = n} f(u,v) = \max_{n/2 \leq u \leq \Delta} f(u,n-u).    
\end{equation}

Now let $u = \Delta'$ be the number that maximizes $f(u,n-u)$ under the constraint $n/2 \leq u \leq \Delta$. We are going to show that $n(S_{a ,b},K_{\Delta',n-\Delta'}) \geq n(S_{a,b},G)$:
\begin{equation}
\label{thm2-heart}
    \begin{aligned}
         n(S_{a,b},G) &= \sum_{uv \in E} \binom{d_G(u) -1}{a}\binom{d_G(v) - 1}{b} + \binom{d_G(u) - 1}{b}\binom{d_G(v) - 1}{a} \\
    &= \sum_{uv \in E} f(d_G(u),d_G(v)) \\ 
    &\leq \sum_{uv \in E} f(\Delta',n-\Delta') \\ 
    &= |E(G)|f(\Delta',n-\Delta') \\ 
    &\leq \Delta'(n-\Delta')f(\Delta',n-\Delta') \\ 
    &= n(S_{a,b},G_{\Delta',n-\Delta'}).
    \end{aligned}
\end{equation}

As $G$ is a triangle-free graph with maximum degree $\Delta$, $d_G(u),d_G(v) \leq \Delta$ and $d_G(u) + d_G(v) \leq n$. Thus, 
\begin{equation}
    \displaystyle f(d_G(u),d_G(v)) \leq \max_{x,y \leq \Delta,x+y\leq n} f(x,y) = \max_{n/2 \leq x \leq \Delta}f(x,n-x) = f(\Delta',n-\Delta'),
\end{equation}
and the inequality in the third line of Equation \ref{thm2-heart} holds. By Lemma \ref{max-edge-triangle-free}, since $G$ has maximum degree $\Delta$, $|E(G)| \leq \Delta(n-\Delta)$. Since $\frac{n}{2} \leq \Delta' \leq \Delta$, we have $|E(G)| \leq \Delta(n-\Delta) \leq \Delta'(n-\Delta')$, and thus the inequality in the fifth line of Equation \ref{thm2-heart} holds. 

Therefore, we conclude that there exists $\Delta'$ such that $\frac{n}{2} \leq \Delta' \leq \Delta$ and $n(S_{a,b},G) \leq n(S_{a,b},G_{\Delta',n-\Delta'})$.

If $G$ is not a complete bipartite graph, by Lemma \ref{max-edge-triangle-free}, we have $|E(G)| < \Delta(n-\Delta) \leq \Delta'(n -\Delta')$. When $\floor{n/2} \geq \max\{a,b\}+1$, 
\begin{equation}
    f(\Delta',n-\Delta') = \max_{u,v \leq \Delta,u+v = n}f(u,v) \geq f(\ceil{n/2},\floor{n/2}) >0.
\end{equation}
Therefore, the fifth line of  Equation \ref{thm2-heart} is a strict inequality and we have $n(S_{a,b},G) < n(S_{a,b},K_{\Delta',n-\Delta'})$.
\end{proof}

The above result tells us we may limit the search of the triangle-free graph with most copies of $S_{a,b}$ with fixed number of vertices to complete bipartite graphs. A natural follow-up question to ask is to determine when the Tur\'an graph $K_{\lfloor \frac{n}{2} \rfloor, \lceil \frac{n}{2} \rceil}$ maximizes the number of copies of $S_{a,b}$. First, we are going to show when $a = b$, then the Tur\'an graph $K_{\lfloor \frac{n}{2} \rfloor, \lceil \frac{n}{2} \rceil}$ maximizes $n(S_{a,b},G)$.

\begin{theorem}
The Tur\'an graph $K_{\lfloor \frac{n}{2} \rfloor, \lceil \frac{n}{2} \rceil}$ maximizes $n(S_{a,a},G)$ assuming that $G$ is triangle-free with $n$ vertices.
\end{theorem}

\begin{proof}
Since we have a symmetric double star $S_{a,a}$, the number of double stars with central edge $uv$ is
\begin{math}
\binom{d_G(u) -1}{a} \binom{d_G(v) - 1}{a}.
\end{math}
Since we may limit the search of graphs maximizing $n(S_{a,a},G)$ to complete bipartite graphs, for a complete bipartite graph $K_{x,n-x}$, we have:
\begin{equation}
n(S_{a,a},K_{x,n-x}) = x(n-x)\binom{x-1}{a}\binom{n-x-1}{a}.
\end{equation}

We know that 
\begin{equation}
    \binom{x-1}{a}\binom{n-x-1}{a} = \frac{1}{a!} \prod_{i=1}^a (x-i)(n-x-i).
\end{equation}
Each term $(x-i)(n-x-i) = -x^2 + nx + (i^2 - in)$ is maximized at $x = \lfloor \frac{n}{2} \rfloor$ and $x = \lceil \frac{n}{2} \rceil$ assuming that $x$ is an integer. Furthermore, $x(n-x)$ is also maximized at  $x = \lfloor \frac{n}{2} \rfloor$ and $x = \lceil \frac{n}{2} \rceil$ assuming that $x$ is an integer. Therefore, we conclude that the Tur\'an graph $K_{\lfloor \frac{n}{2} \rfloor, \lceil \frac{n}{2} \rceil}$ maximizes $n(S_{a,a},G)$ assuming that $G$ is triangle-free with $n$ vertices.
\end{proof}

Furthermore, we will show that the Tur\'an graph $K_{\lfloor \frac{n}{2} \rfloor, \lceil \frac{n}{2} \rceil}$ maximizes $n(S_{1,3},G)$ but not $n(S_{1,4},G)$:

\begin{theorem}
  The Tur\'an graph $K_{\lfloor \frac{n}{2} \rfloor, \lceil \frac{n}{2} \rceil}$ maximizes $n(S_{1,3}, K_3)$ when $n \geq 13$.
\end{theorem}

\begin{proof}
This follows from a direct computation. Let
\begin{equation}
    f(n,x) := n(S_{1,3},K_{x,n-x}) = x(n-x) \left( (x-1)\binom{n-x-1}{3} + (n-x-1)\binom{x-1}{3} \right).
\end{equation}

Now it follows that
\begin{equation}
f(n,x+1)-f(n,x) = \frac{x}{3}(n-1-2x)(n-1-x) \left[ 3x^2 + (3-3n)x + (n^2 - 6n + 14) \right].
\end{equation}

If $\lceil \frac{n}{2} \rceil \leq x < n-1$, we have $n - 1 - 2x < 0$ and $n - 1 - x > 0$. We know that the discriminant for $3x^2 + (3-3n)x + (n^2 - 6n + 14)$ is 
\begin{equation}
    \Delta = (3-3n)^2 - 4 \cdot 3 \cdot (n^2 - 6n+14) = -3(n^2 - 18n+ 53).
\end{equation}
As a result, if $\Delta < 0$, we must have $3x^2 + (3-3n)x + (n^2 - 6n+14) > 0$ for all $x$, and therefore if $\Delta < 0$ then we must have $f(n,x+1) - f(n,x) < 0$ for all $\lceil \frac{n}{2} \rceil \leq x < n-1$. Solving $\Delta < 0$ yields $n > 9+2\sqrt{7} \approx 14.29$, and if $n = 13$ or $n = 14$ we can manually verify that the Tur\'an graphs maximize $n(S_{1,3},G)$. When $n = 13$, we have $n(S_{1,3},K_{6,7}) = n(S_{1,3},K_{5,8}) = 6720$. When $n = 14$, we have $n(S_{1,3},K_{7,7}) = n(S_{1,3},K_{6,8}) = 11760$.
\end{proof}

\begin{theorem}
Suppose $x_0$ maximizes $n(S_{1,4},K_{x,n-x})$ over $[\frac{n}{2},n-1]$. Then $\displaystyle \lim_{n \rightarrow \infty} \frac{x_0}{n} = \frac{2}{3}$. In other words, if the number of vertices $n$ approaches infinity, the complete bipartite graph $K_{x,n-x}$ maximizing $n(S_{1,4},G)$ contains two parts with approximately $\frac{2}{3}n$ and $\frac{1}{3}n$ vertices each.
\end{theorem}

\begin{proof}
This also comes directly from computation. Let
\begin{equation}
f(n,x) := n(S_{1,4},K_{x,n-x}) = x(n-x)\left( (x-1)\binom{n-x-1}{4} + (n-x-1)\binom{x-1}{4} \right).
\end{equation}

Now it follows that
\begin{equation}
f(n,x+1) - f(n,x) = \frac{x}{24}(n-6)(n-1-2x)(n-1-x)\left[ 9x^2 + (9-9n)x + (2n^2 - 9n + 22) \right].
\end{equation}

Assuming that $n > 6$, if $\lceil \frac{n}{2} \rceil \leq x < n-1$, we must have $n-6 > 0$, $n-1-2x < 0$, $n-1-x > 0$. We need to look at $9x^2 + (9-9n)x + (2n^2 - 9n+22)$. The discriminant is 
\begin{equation}
    \Delta = (9-9n)^2 - 36(2n^2 - 9n+22) = 9(n^2 + 18n - 79),
\end{equation}
and therefore we conclude that if $n > 4\sqrt{10} - 9 \approx 3.65$, $\Delta > 0$ and there are two real roots for $9x^2 + (9-9n)x + (2n^2 - 9n+22) = 0$. The roots are
\begin{equation}
x_0 = \frac{3n - 3 \pm \sqrt{n^2 + 18n - 79}}{6}.
\end{equation}

We will discard the root $x = \frac{3n - 3 - \sqrt{n^2 + 18n-79}}{6} < \frac{n}{2}$ for our purpose and let $x_0 = \frac{3n - 3 + \sqrt{n^2 + 18n-79}}{6}$. Now it follows assuming that $n > 6$, we have $f(n,x+1) - f(n,x) < 0$ if $\lceil \frac{n}{2} \rceil \leq x < x_0$ and $f(n,x+1) - f(n,x) > 0$ if $x_0 < x < n-1$.

As we have $\displaystyle \lim_{n \rightarrow \infty}\frac{x_0}{n} = \frac{2}{3}$, we may conclude if the number of vertices $n$ approaches infinity, the complete bipartite graph $K_{x,n-x}$ maximizing $S_{1,4}(G)$ contains two parts with approximately $\frac{2}{3}n$ and $\frac{1}{3}n$ vertices each.
\end{proof}

We might want to consider what is going to happen should the number of vertices in the double star increases, and when the Tur\'an graph achieves $\Extremal(n,S_{a,b},K_3)$. Unfortunately, when $a$ and $b$ get larger, it is more difficult to study the behavior of 
\begin{equation}
n(S_{a,b},K_{x,n-x}) = x(n-x) \left( \binom{x-1}{a}\binom{n-x-1}{b} + \binom{x-1}{b}\binom{n-x-1}{a} \right).
\end{equation} Nevertheless, we may approximate \begin{math}
n(S_{a,b},K_{x,n-x})
\end{math}
with 
\begin{equation}
    x(n-x)(x^a(n-x)^b + x^b(n-x)^a) = [x(n-x)]^{a+1} \left( x^{b-a} + (n-x)^{b-a} \right).
\end{equation} Since $x(n-x)$ is decreasing over $\frac{n}{2} \leq x \leq n$ while $x^{b-a} + (n-x)^{b-a}$ is increasing over $\frac{n}{2} \leq x \leq n$, when the difference $b-a$ is large compared to $a+1$, we expect $x^{b-a} + (n-x)^{b-a}$ dominates and the extremal graph will be a more unbalanced complete bipartite graph. Here is a table of the $x \in [\frac{1}{2},1]$ that maximizes $x^a(1-x)^b + x^b(1-x)^a$. 

\begin{center}
\begin{tabular}{c|ccccccccc}
    $b = $ & 1 & 2 & 3 & 4 & 5 & 6 & 7 & 8 & 9 \\
    \hline
    $a = 1$ & $\frac{1}{2}$ & $\frac{1}{2}$ & $\frac{1}{2}$ & $\frac{3+\sqrt{3}}{6}$ & .832 & .857 & .875 & .889 & .900  \\ 
    $a = 2$ & & $\frac{1}{2}$ & $\frac{1}{2}$ & $\frac{1}{2}$ & $\frac{2}{3}$ & .743 & .777 & .800 & .818 \\ 
    $a= 3$ & & & $\frac{1}{2}$ & $\frac{1}{2}$ & $\frac{1}{2}$& $\frac{1}{2}$ & .682 & $\frac{5+\sqrt{5}}{10}$ & .749 \\ 
    $a = 4$ & & & & $\frac{1}{2}$ & $\frac{1}{2}$ & $\frac{1}{2}$ & .633 & .684 & .712 \\ 
    $a = 5$ & & & & & $\frac{1}{2}$ & $\frac{1}{2}$ & $\frac{1}{2}$ & $\frac{1}{2}$ & .585 \\ 
    $a = 6$ & & & & & & $\frac{1}{2}$ & $\frac{1}{2}$ & $\frac{1}{2}$ & $\frac{1}{2}$ \\
\end{tabular}
\end{center}

From this table, we see when $a$ is fixed and $b \geq a$, then the $x_{\max}$ that maximizes $x^a(1-x)^b + x^b(1-x)^a$ over $[\frac{1}{2},1]$ will not decrease. We also see if $b$ is sufficiently large compared to $a$, then $x_{\max} \approx \frac{b}{a+b}$. We know that $x^b(1-x)^a$ is maximized at $\frac{b}{a+b}$, so this result should not be too surprising.

\section{Adjacent Vertex Condition}
In the double star problem we have just studied, the forbidden subgraph is the triangle and we are maximized the number of double stars. One possible way to invert the problem is to put conditions on the double stars in the graph and count the number of triangles.

This leads us to the ``adjacent vertex condition." In this problem the forbidden subgraphs $G$  are double stars $S_{a,b}$ for all $a,b$ such that $a+b = k$ for some fixed $k$. This is equivalent to saying $d_G(x) + d_G(y) < k + 2$ for all $xy \in E(G)$. Therefore, a related problem is if for every $xy \in E(G)$ the degree sum is large (\textit{i.e.} $d_G(x) + d_G(y) \geq k$ for some $k$) what the minimum number of triangles in the graph will be.

If for all $xy \in E(G)$, $d_G(x) + d_G(y) \geq k$ for some $k \leq n$, then the graph can be triangle-free. Any complete bipartite graph $K_{a,b}$ with $n$ vertices will be triangle-free and satisfies $d_G(x) + d_G(y) = n$ for all $xy \in K_{a,b}$. Thus, a more interesting question is what will happen if $d_G(x) + d_G(y) \geq n + k$ for some constant $k > 0$. We are going to prove if $n$ is sufficiently large compared to $k$, then such a graph will contain at least $\binom{k+1}{2}(n-k-1)+\binom{k+1}{3}$ triangles.\\

\begin{theorem}
\label{thm:adjacent}
Let $G$ be a connected graph with $n \geq 6(k+1)(k+2)$ vertices and for any $xy \in E(G)$, $d_G(x) + d_G(y) \geq n + k$. Then $G$ contains at least $\binom{k+1}{2}(n-k-1)+\binom{k+1}{3}$ triangles.
\end{theorem}

To prove this theorem, we will need the following lemma:

\begin{lemma}
Suppose the number of vertices with degree less than $\frac{n+k}{2}$ is $m$ and the minimal degree in the graph is $\delta$. Then the graph contains at least $\frac{k\delta}{2}m$ triangles.
\end{lemma}

\begin{proof}
Let $m$ denote the number of vertices with degree less than $\frac{n+k}{2}$. First, notice that vertices with degree less than $\frac{n+k}{2}$ cannot be connected to each other and thus they form an independent set. This also means any edge incident to vertex with degree less than $\frac{n+k}{2}$ must be connected to a vertex with degree at least $\frac{n+k}{2}$. Now let us try to count the number of triangles sitting ``in between" vertices with small degree and vertices with large degree, and temporarily neglect the triangles whose three vertices are the vertices with large degree. Here, ``large degree" means vertices $v$ such that $d_G(v) \geq \frac{n+k}{2}$ while ``small degree" means $d_G(v) < \frac{n+k}{2}$.

Let us consider the sum $\displaystyle \sum_{e} (\text{\# triangles containing }e)$ for all $e$ such that one of the endpoint has ``small" degree and one of the endpoint has ``large" degree. Since we know the vertices with small degree form an independent set, a triangle can only contain three vertices of large degree or two vertices of large degree and a vertex of small degree. Triangles containing three vertices of large degree will not have any contribution to the sum, while triangles
such that two of the vertices have large degree and one of the vertices has small degree will contribute 2 to the sum since two of the triangle's edges are between vertices with small degree and vertices with large degree. Thus, we conclude the sum is essentially twice the number of triangles sitting in between vertices of large degree and vertices of small degree, and thus we know that $\displaystyle t_3(G) \geq \frac{1}{2}\sum_{e}(\text{\# triangles containing }e)$ where $t_3(G)$ is the number of triangles in the graph and the sum is over all edges in between vertices of small degree and vertices of large degree. By the adjacent vertex condition, the number of triangles containing $e$ for an arbitrary edge $e \in G$ is $k$. Since the minimum degree is $\delta$ and edges from the vertices with small degree shall go to vertices with large degree, we know that the number of edges such that one endpoint is a vertex with small degree and one vertex is a vertex with large degree is at least $m \delta$.

Therefore, $t_3(G) \geq \frac{m\delta}{2}k$, as desired.
\end{proof}

With the above lemma, we are able to show if the number of vertices is sufficiently large, $G$ is connected, and for any $xy \in E(G)$, $d_G(x) + d_G(y) \geq n +k$, then $G$ contains at least $\binom{k+1}{2}(n-k-1) + \binom{k+1}{3}$ triangles.

\begin{proof}[of Theorem \ref{thm:adjacent}]
The optimal graph is constructed as follows: we construct a clique of $k+1$ vertices and the rest $n-k-1$ vertices are connected to the vertices in the clique. In this way, the $k+1$ vertices in the clique have degree $n-1$ and the other $n-k-1$ vertices have degree $k+1$.

The first step is to show we cannot achieve a lower number of triangles with a graph with minimum degree $\delta \geq k + 2$ by a case analysis on $m$, the number of vertices with degree less than $\frac{n+k}{2}$.

The first case is when many of vertices of $G$ have degree less than $\frac{n+k}{2}$. By the above lemma, if $m \geq \frac{k+1}{k+2}n$ and $\delta \geq k+2$, then $t_3(G) \geq \binom{k+1}{2}n$ because  
\begin{equation}
  t_3(G) \geq \frac{m\delta}{2}k \geq \frac{1}{2}\left( \frac{k+1}{k+2}n \right)(k+2)k = \frac{(k+1)k}{2}n = \binom{k+1}{2}n.  
\end{equation}
The optimal graph we constructed has 
\begin{equation}
    \binom{k+1}{2}(n-k-1) + \binom{k+1}{3} = \binom{k+1}{2}n - \left[ \binom{k+1}{2}(k+1) - \binom{k+1}{3} \right]
\end{equation} triangles. However, 
\begin{equation}
    \binom{k+1}{2}(k+1) - \binom{k+1}{3} \geq 0
\end{equation}
and thus 
\begin{equation}
    \binom{k+1}{2}n \geq \binom{k+1}{2}(n-k-1) + \binom{k+1}{3},
\end{equation}
because we know that 
\begin{math}
    \binom{k+1}{2}(k+1) = \frac{(k+1)^2k}{2}
\end{math} while 
\begin{math}
    \binom{k+1}{3} = \frac{(k+1)k(k-1)}{6}.
\end{math}
Since we have $k \geq 1$, we know that $(k+1)^2k$ and $(k+1)k(k-1)$ are all nonnegative and by factor-by-factor comparison we see 
\begin{equation}
    (k+1)^2k \geq (k+1)k(k-1).
\end{equation}
Since $2 \leq 6$, we see $\frac{(k+1)k^2}{2} \geq \frac{(k+1)k(k-1)}{6}$ and thus 
\begin{equation}
    \binom{k+1}{2}(k+1) - \binom{k+1}{3} \geq 0.
\end{equation}
As a result, we know that if $m \geq \frac{k+1}{k+2}{n}$ and $\delta \geq k+2$ then $t_3(G)$ cannot be less than the optimal graph we have constructed.

To solve the case when $m < \frac{k+1}{k+2}{n}$, notice that in this case $n - m$ is large and thus there are a lot of large-degree vertices. We may use a rough estimate based on the number of edges in the graph, assuming that $n$ is large compared to $k$. When $m < \frac{k+1}{k+2}{n}$, we have $n - m > \frac{1}{k+2}{n}$. Now the degree sum of $G$ will be at least $(n-m)\frac{n+k}{2} > \frac{n(n+k)}{2(k+2)}$. Therefore, the number of edges is greater than $\frac{n(n+k)}{4(k+2)}$, and thus we know the number of triangles is at least 
\begin{equation}
\frac{k}{3} \frac{n(n+k)}{4(k+2)} = \frac{k}{12(k+2)}\left[ n(n+k) \right].
\end{equation}
In particular, suppose $n \geq 6(k+1)(k+2)$, we have 
\begin{equation}
\begin{aligned}
\frac{k}{12(k+2)}n^2
&\geq
\frac{k}{12(k+2)} [6(k+1)(k+2) n] \\
&= \frac{k(k+1)}{2} n \\ 
&= \binom{k+1}{2}n \\
&\geq \binom{k+1}{2}(n-k - 1) + \binom{k+1}{3}
\end{aligned}
\end{equation}
and thus 
\begin{equation}
    \frac{k}{12(k+2)}[n(n+k)] \geq \binom{k+1}{2}(n-k-1) + \binom{k+1}{3}.
\end{equation}

Therefore, based on the arguments above, we can show when $n \geq 6(k+1)(k+2)$ and when $\delta \geq k+2$,
\begin{equation}
    t_3(G) \geq \binom{k+1}{2}n  \geq \binom{k+1}{2}(n-k-1)+\binom{k+1}{3}.
\end{equation}
Suppose the number of vertices with degree less than $\frac{n+k}{2}$ is $m$, then if $m \geq \frac{k+1}{k+2}n$ then we conclude that the number of triangles is at least $\binom{k+1}{2}n$, based on the argument in the first half. If on the other hand, $m < \frac{k+1}{k+2}n$, we still know that the number of triangles is at least $\binom{k+1}{2}n$, based on the argument in the second half. Thus, the extremal graph cannot have a minimum degree $\delta \geq k+2$. Assuming that our graph is connected (so we do not have vertices with degree 0), and since if a vertex is not an isolated vertex its degree must be at least $k+1$ (otherwise its neighbor will have degree at least $n$, which cannot happen), we conclude the optimal graph has minimum degree $k+1$.

Finally, suppose our graph $G$ has minimum degree $k+1$, and we are going to show $G$ must contain at least $\binom{k+1}{2}(n-k-1)+ \binom{k+1}{3}$ triangles. Let $x$ be a vertex with degree $k+1$ in $G$. By the adjacent vertex condition, the $k+1$ neighbors of $x$ must have degree $n-1$, and thus $k+1$ neighbors of $x$ are connected to all other vertices in the graph. Therefore, if $G$ has minimum degree $k+1$, the optimal graph we have constructed must be a subgraph of $G$. Since adding more edges will not decrease the number of triangles, we conclude $G$ must contain at least $\binom{k+1}{2}(n-k-1)+ \binom{k+1}{3}$ triangles in this case.

\end{proof}

In fact, the extremal construction above also minimizes the number of edges, assuming that there are no isolated vertex in the graph and the number of vertices are sufficiently large:

\begin{theorem}
Suppose $G$ is a connected graph with $n \geq 2k+2$ vertices, and suppose for all $xy \in E(G)$, $d_G(x) + d_G(y) \geq n + k$. Then $G$ contains at least $(k+1)n - \binom{k+2}{2}$ edges.
\end{theorem}

\begin{proof}
Let the vertex with minimum degree in $G$ be $u$, and assume that its degree is $d_G(u) = \delta$. Since $G$ is connected, $\delta > 0$. As a result, by our argument above, the adjacent vertex condition requires $\delta \geq k + 1$. Now it follows the $\delta$ vertices that are connected with $u$ must have degree at least $\max \{n+k-\delta, \delta\}$, and the rest $n-\delta$ vertices clearly have degrees at least $\delta$. Therefore, the sum of degrees is at least $f(\delta) = \delta \max\{n+k-\delta, \delta\} + (n-\delta)\delta$.

If $\delta \leq \frac{n+k}{2}$, we have $\max\{n+k - \delta,\delta\} = n+k-\delta$, and thus 
\begin{equation}
    f(\delta) = \delta(n+k-\delta) + (n-\delta)\delta = \delta(k+2n-2\delta).
\end{equation}
Since the coefficient for the quadratic term in $f(\delta)$ is negative, the minimum of $f(\delta)$ over interval $[k + 1, \frac{n+k}{2}]$ can only be taken at either $k+1$ or $\frac{n+k}{2}$. Since \begin{equation}
    f(\frac{n+k}{2}) - f(k+1) = \frac{1}{2}(n-k-2)(n-2k-2),
\end{equation}
we know assuming that $n \geq 2k+2$, $f(\frac{n+k}{2}) - f(k+1) \geq 0$ and therefore $\delta = k+1$ minimizes $f(\delta)$. In this case, 
\begin{equation}
f_{\min}(\delta) = (2n-k-2)(k+1).
\end{equation}

If $\delta > \frac{n+k}{2}$, then we have $f(\delta) = n\delta > n\frac{n+k}{2}$. Since 
\begin{equation}
    \frac{n+k}{2} - (2n-k-2)(k+1) = \frac{1}{2}(n-k-2)(n-2k-2) \geq 0,
\end{equation}
when $n \geq 2k+2$, we conclude that $f$ is minimized at $\delta = k+1$, and the degree sum is at least $(2n-k-2)(k+1)$. Therefore, the number of edges in $G$ is at least 
\begin{equation}
    \frac{(2n-k-2)(k+1)}{2} = (k+1)n - \binom{k+2}{2}.
\end{equation}

However, the extremal graph constructed above has \begin{equation}
    \binom{k+1}{2} + (n-k-1)(k+1) = (k+1)n - \binom{k+2}{2}
\end{equation}
edges, so $ (k+1)n - \binom{k+2}{2}$ edges can be achieved.
\end{proof}

This result should not be too surprising: the adjacent vertex condition essentially guarantees that there are at least $k$ triangles containing a given edge, and so to minimize the number of triangles, we might want to minimize the number of edges first.

\section{Non-adjacent vertices condition}
A even more interesting question is the minimum number of edges or triangles in a graph such that two \emph{non-adjacent} vertices have large degree sum. One classical result with this type of degree sum condition is the Ore's theorem, which states that if a graph $G$ satisfies $d_G(x) + d_G(y) \geq n$ for all $xy \notin E(G)$, $x \neq y$, then the graph $G$ is Hamiltonian. Also, conditions of the form $d_G(x) + d_G(y) \geq k$ for all $xy \notin E(G), x \neq y$ mean the complement of the graph does not contain big double stars. 

The particular non-adjacent vertex condition we will study is the condition that $d_G(x) + d_G(y) \geq n+1$ for all $xy \notin E(g), x \neq y$. As we are going to show next, the minimum number of triangles in the graph will depend on the parity of the number of vertices. To prove the lower bound on the number of triangles for a graph satisfying the non-adjacent vertex condition, we will first establish tight lower bounds on the number of edges.

\subsection{Number of Edges}
We have the following result on the minimum number of edges a graph satisfying the non-adjacent vertex condition must contain:

\begin{theorem} For a graph $G$ of order $n \geq 3$ that satisfies the condition $d_G(x) + d_G(y) \geq n+1$ for all $xy \notin E(G)$, $x \neq y$, the minimum number of edges in $G$ are:
\begin{equation}
    |E(G)| \geq \left\{
                \begin{array}{ll}
                  4k^2 - k & n =4k-1,\\
                  4k^2 + k + 1 & n = 4k,\\
                  4k^2 + 3k + 1 & n = 4k+1,\\
                  4k^2 + 5k + 2 & n = 4k+2.\\
                \end{array}
              \right.
\end{equation}

In other words,
\begin{equation}
    |E(G)| \geq \frac{n^2}{4} + \frac{n}{4} + \left\{
                \begin{array}{ll}
                  0 & (n \equiv3 \bmod 4), \\
                  \frac{1}{2} & (n\equiv 1,2 \bmod 4), \\
                  1 & (n \equiv 0 \bmod4).
                \end{array}
              \right.
\end{equation}
\end{theorem}

Thus, we know that the minimum number of edges in a graph satisfying our condition must contain $\frac{n^2}{4} + \frac{n}{4} + c$ edges where the constant $0 \leq c \leq 1$.

\begin{proof}
Let $G$ be a graph of order at least $3$ that satisfies the condition $d_G(x) + d_G(y) \geq n+1$ for all $xy \notin E(G)$ such that $x \neq y$. Consider a vertex $v$ of minimal degree $j$ in $G$. The $n-j-1$ vertices in $G$ that are not $v$ or neighbors of $v$ must have degree at least $n+1-j$. Thus,
\begin{equation}
    \begin{aligned}
          |E(G)| &\geq \frac{j(j+1) + (n-1-j)(n+1-j)}{2}\\
  &= \frac{2j^2 - (2n-1)j + n^2-1}{2}. 
    \end{aligned}
\end{equation}

When $j = \frac{2n-1}{4}$, $\frac{2j^2 - (2n-1)j + n^2-1}{2} $ attains its minimum. Thus,
\begin{equation}
\label{uniform-lower-bound}
    |E(G)| \geq \frac{1}{16}(4n^2 + 4n-9).
\end{equation}
This is a universal lower bound for the minimum number of edges in $G$. The general strategy for constructing graphs satisfying the non-adjacent vertex condition and having  minimum number of edges is to construct graphs that are close to regular graphs (\textit{i.e.} the degrees of the vertices are similar).

When the number of vertices is odd, say $n = 2k+1$ for some $k$, then we see a $(k+1)$-regular graph will satisfy the non-adjacent vertex condition as for two arbitrary vertices in the graph we have $d_G(x) + d_G(y) = 2k+2 = n+1$. We distinguish four cases on the remainder of $n$ modulo 4.\\

\begin{enumerate}[label=Case \arabic*:]
\item $n = 4k-1$ for some $k \in \naturals$.

Any $2k$-regular graph will satisfy the non-adjacent vertex condition. The graph contains 
\begin{equation}
\frac{1}{2}(2k)(4k-1) = 4k^2 - k = \frac{n^2 + n}{4}
\end{equation}
edges, and it is optimal since we have shown that $|E(G)| \geq 4k^2 - k - \frac{9}{16}$ and as $|E(G)|$ is an integer we must also have $|E(G)| \geq 4k^2 - k$. The $2k$-regular graph we are going to construct contains few triangles. We begin with the complete bipartite graph $K_{2k,(2k-1)}$. Inside the part with $2k$ vertices, we add a perfect matching.  In the resulting graph, each vertex will have degree $2k = \frac{n+1}{2}$.
\begin{figure}[h!]
    \centering

    \begin{tikzpicture}
        \draw (15,1)-- (16,1);
        \draw (17,1)-- (18,1);
        \foreach \i in {15,16,17,18}{
            \foreach \j in {15.5,16.5,17.5}{
                \draw (\i,1)-- (\j,0);
            }
        }

        \begin{scriptsize}
            \draw [fill=black] (15,1) circle (3.5pt);
            \draw [fill=black] (16,1) circle (3.5pt);
            \draw [fill=black] (17,1) circle (3.5pt);
            \draw [fill=black] (18,1) circle (3.5pt);
            \draw [fill=black] (15.5,0) circle (3.5pt);
            \draw [fill=black] (16.5,0) circle (3.5pt);
            \draw [fill=black] (17.5,0) circle (3.5pt);
        \end{scriptsize}
    \end{tikzpicture}
        \caption{This is an example of the minimal graph when $n=4 \times 2-1 = 7$}
\end{figure}
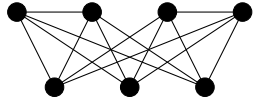

\item $n = 4k+1$ for some $k \in \naturals$.

In this case, we cannot construct a $(2k+1)$-regular graph since both $n$ and $2k+1$ are odd, but it is possible to construct a graph where all except one vertex have degree $2k+1$ and one vertex has $2k+2$. The resulting graph contains $4k^2 + 3k + 1$ edges. For two distinct vertices $x,y$ of the graph, we have $d_G(x) + d_G(y) = n + 1$ or $n+2$. Our construction begins with the complete bipartite graph $K_{2k,(2k+1)}$ and inside the part with $2k+1$ vertices, we add a semi-matching: we add a perfect matching among first $2k$ vertices and the last vertex will be connected to one of the first $2k$ vertices. Vertices in the part with $2k$ vertices will have degree $2k+1$, while the vertices in the part with $2k+1$ vertices will have degree $2k+1$ or $2k+2$. 
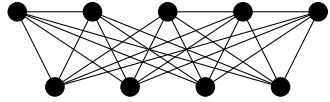
\begin{figure}[h!]
    \centering

    \begin{tikzpicture}
        \draw (14.5,1)-- (15.5,1);
        \draw (16.5,1)-- (18.5,1);
        \foreach \i in {15,16,17,18}{
            \foreach \j in {14.5,15.5,16.5,17.5,18.5}{
                \draw (\i,0)-- (\j,1);
            }
        }

        \begin{scriptsize}
            \draw [fill=black] (15,0) circle (3.5pt);
            \draw [fill=black] (16,0) circle (3.5pt);
            \draw [fill=black] (17,0) circle (3.5pt);
            \draw [fill=black] (18,0) circle (3.5pt);
            \draw [fill=black] (14.5,1) circle (3.5pt);
            \draw [fill=black] (15.5,1) circle (3.5pt);
            \draw [fill=black] (16.5,1) circle (3.5pt);
            \draw [fill=black] (17.5,1) circle (3.5pt);
            \draw [fill=black] (18.5,1) circle (3.5pt);
        \end{scriptsize}
    \end{tikzpicture}
    \caption{This is an example of the minimal graph when $n=4 \times 2+1 = 9$}
\end{figure}

To see $4k^2 + 3k + 1$ edges is optimal, notice that by Equation \ref{uniform-lower-bound},
\begin{equation}
    |E(G)| \geq \frac{1}{16}(4n^2 + 4n - 9) = 4k^2 + 3k - \frac{1}{16}
\end{equation}
automatically gives us $|E(G)| \geq 4k^2 + 3k$ since $|E(G)|$ is an integer. To see that we cannot obtain $|E(G)| = 4k^2 + 3k$, if the minimum degree is not equal to $2k$, then 
\begin{equation}
|E(G)| \geq 4k^2 + 3k + \frac{1}{2}.
\end{equation}
Even if the minimum degree is equal to $2k$, to achieve $|E(G)| = 4k^2 + 3k$, we must make sure the $2k+1$ vertices that are either the vertex with minimum degree itself or a neighbor to the vertex with minimum degree have degree $2k$ and the rest $2k$ vertices have degree $2k+2$ based on the analysis at the beginning. This is impossible: the non-adjacent vertex condition requires all vertices of degree $2k$ to be connected to each other, and thus the $2k+1$ vertices of degree $2k$ must form a clique and no other edges may be connected to those $2k+1$ vertices. However, since the rest $2k$ vertices must have degree $2k+2$, we need an edge that goes from a vertex of degree $2k+2$ to a vertex of degree $2k$, which is impossible.

\end{enumerate}

The construction is different for graphs with even number of vertices. If the number of vertices in the graph is $n  =2k$, then a $k$-regular graph will violate the non-adjacent vertex condition, while a $(k+1)$-regular graph will contain too many edges.  Vertices with degree at most $k$ must be connected to each other, so in the construction below we are going to start with a disjoint union of two cliques.

\begin{enumerate}[resume*]
\item $n = 4k+2$ for some $k \in \naturals$.

In this case, we can use the following construction: We begin with a disjoint union of cliques $K_{2k}$ and $K_{2k+2}$. Then, we are going to add and remove some edges so that eventually vertices originally in the clique $K_{2k}$ will have degree $2k+1$ and vertices originally in the clique $K_{2k+2}$ will have degree $2k+2$. The total number of edges is $4k^2 + 5k +2$, and vertices with degree $2k+1$ are connected to each other.

Let us label the vertices in $K_{2k}$  $u_1,\ldots,u_{2k}$ and the vertices inside $K_{2k+2}$ $v_1,\ldots,v_{2k+2}$. We first add some edges going between the two cliques so that vertices $u_1,\ldots,u_{2k}$ will have degree $2k+1$ in the end. For each $u_i$ ($1 \leq i \leq 2k-2)$, we connect $u_i$ to $v_j$ and $v_{j+1}$. For the last two vertices $u_{2k-1}$ and $u_{2k}$ we connect $u_i (i = 2k-1, 2k)$ with $v_i$ and $v_{i+2}$. We see now vertices $v_1,v_{2k},v_{2k+1},v_{2k+2}$ have degree $2k+2$ while vertices $v_2,\ldots,v_{2k-1}$ will have degree $2k+3$, so we are going to remove a perfect matching and reduce the degrees of $v_2,\ldots,v_{2k-1}$ by 1: we can simply remove the matching $v_2v_3, v_4v_5, \ldots, v_{2k-2}v_{2k-1}$. 

To see that we cannot do better than $4k^2 + 5k + 2$ edges, notice that Equation \ref{uniform-lower-bound} implies 
\begin{equation}
    |E(G)| \geq 4k^2 + 5k + \frac{15}{16}.
\end{equation}
Since $|E(G)|$ is always an integer, 
\begin{equation} 
|E(G)| \geq 4k^2 + 5k + 1.
\end{equation}
However, if the minimum degree is not equal to $2k+1$, $|E(G)| > 4k^2 + 5k + 1$ by the analysis at the beginning. If the minimum degree is $2k+1$, let the vertex with minimal degree be $v$, then we see we may only achieve equality when the $2k$ vertices that are neither $v$ itself or adjacent to $v$  have degree $2k+2$ and the rest $2k+2$ vertices have degree $2k+1$. This is not possible because the $2k+2$ vertices with degree $2k+1$ must form a clique and there cannot be additional edges connected to those $2k+2$ vertices, while to make sure the rest $2k$ vertices have degree $2k+2$ there must exist some edges between vertices with degree $2k+1$ and $2k+2$.

\begin{figure}[h!]
    \centering

    \begin{tikzpicture}
        \foreach \i in {14,15}{
                \draw (\i,1)-- (\i,0);
                \draw (\i,1)-- (\i+1,0);
        }
        \draw(16,1) -- (16,0);
        \draw(16,1) -- (18,0);
        \draw(17,1) -- (17,0);
        \draw(17,1) -- (19,0);

        \draw[dashed](15,0)--(16,0);
        \begin{scriptsize}
            \draw [fill=black] (14,0) circle (2pt
            )node[anchor=north]{$v_1$};
            \draw [fill=black] (15,0) circle (2pt) node[anchor=north]{$v_2$};
            \draw [fill=black] (16,0) circle (2pt)node[anchor=north]{$v_3$};
            \draw [fill=black] (17,0) circle (2pt)node[anchor=north]{$v_4$};
            \draw [fill=black] (18,0) circle (2pt)node[anchor=north]{$v_5$};
            \draw [fill=black] (19,0) circle (2pt)
            node[anchor=north]{$v_6$};
            \draw [fill=black] (14,1) circle (2pt) node[anchor=south]{$u_1$};
            \draw [fill=black] (15,1) circle (2pt)node[anchor=south]{$u_2$};
            \draw [fill=black] (16,1) circle (2pt)node[anchor=south]{$u_3$};
             \draw [fill=black] (17,1) circle (2pt)node[anchor=south]{$u_4$};
        \end{scriptsize}
    \end{tikzpicture}
    \caption{A minimal graph when $n=4 \times 2+2=10$ (the cliques are omitted)}
\end{figure}
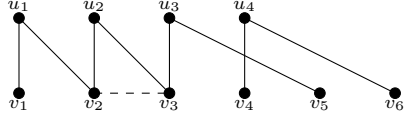

\item $n = 4k$ for some $k \in \naturals$.

In this case, we may construct a graph with $4k^2 + k + 1$ edges. We are going to construct a graph with $2k-2$ vertices having degree $2k$ and $2k+2$ vertices having degree $2k+1$, with the extra restriction that vertices with degree $2k$ are connected to each other. Thus, we can begin with a disjoint union of cliques $G_1 = K_{2k-2}$ and $G_2 = K_{2k+2}$. We will label the vertices in $G_1 = K_{2k-2}$  $u_0,u_1,\ldots,u_{2k-3}$ and vertices in $G_2 = K_{2k+2}$ $v_0,v_1,\ldots,v_{2k+1}$. We need to make sure for each each $u_i$, there are three additional edges adjacent to $u_i$ while degrees of $v_i$ stay the same. Thus, we will add some edges between vertices of $G_1$ and $G_2$ while remove certain edges belonging to $G_2$. Since the edges belong to $G_1$ are not deleted in this process, we know the $2k-2$ vertices originally belong to $G_1$ will be connected to each other, and thus the final graph will satisfy the non-adjacent vertex condition.

If $k \geq 3$,  we can proceed as follows, similar to our approach when $n = 4j + 2$ for some $j$. First, to add edges adjacent to $u_i$ where $i = 0,\ldots, 2k-3$, we will connect $u_i$ with $v_{(i+j)\mod(2k-2)}$ for $j = 0,1,2$. Notice that when $k \geq 3$ (so $2k-2 \geq 4$), $i,i+1,i+2$ are distinct modulo $2k-2$,and thus we will not introduce parallel edges. Now we know the degrees of $v_0,\ldots,v_{2k-3}$ have been increased by 3, and thus we are going to remove a 3-regular subgraph with vertices $v_0,\ldots,v_{2k-3}$. Such a 3-regular graph must exist since $2k-2 \geq 4$ and $2k-2$ is even.

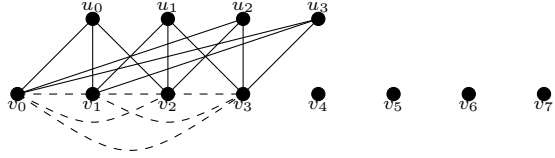
\begin{figure}[h!]
    \centering
    \begin{tikzpicture}
        \foreach \i in {14,15,16,17}{
                \draw (\i,1)-- (\i-1,0);
        }
        \foreach \i in {14,15,16}{
                \draw (\i,1)-- (\i,0);
        }
        \foreach \i in {14,15}{
                \draw (\i,1)-- (\i+1,0);
        }
        \draw(16,1) -- (13,0);
        \draw(17,1) -- (13,0);
        \draw(17,1) -- (14,0);

        \draw[dashed](13,0)--(14,0);
        \draw[dashed](14,0)--(15,0);
        \draw[dashed](15,0)--(16,0);
        \draw[dashed] (13,0) .. controls (14,-0.5) .. (15,0);
        \draw[dashed] (14,0) .. controls (15,-0.5) .. (16,0);
        \draw[dashed] (13,0) .. controls (14.5,-1) .. (16,0);

        \begin{scriptsize}
        \draw [fill=black] (13,0) circle (2.5pt
            )node[anchor=north]{$v_0$};
            \draw [fill=black] (14,0) circle (2.5pt
            )node[anchor=north]{$v_1$};
            \draw [fill=black] (15,0) circle (2.5pt)node[anchor=north]{$v_2$};
            \draw [fill=black] (16,0) circle (2.5pt)node[anchor=north]{$v_3$};
            \draw [fill=black] (17,0) circle (2.5pt)node[anchor=north]{$v_4$};
            \draw [fill=black] (18,0) circle (2.5pt)node[anchor=north]{$v_5$};
            \draw [fill=black] (19,0) circle (2.5pt)node[anchor=north]{$v_6$};
            \draw [fill=black] (20,0) circle (2.5pt)node[anchor=north]{$v_7$};
            \draw [fill=black] (14,1) circle (2.5pt) node[anchor=south]{$u_0$};
            \draw [fill=black] (15,1) circle (2.5pt)node[anchor=south]{$u_1$};
            \draw [fill=black] (16,1) circle (2.5pt)node[anchor=south]{$u_2$};
             \draw [fill=black] (17,1) circle (2.5pt)node[anchor=south]{$u_3$};
        \end{scriptsize}
    \end{tikzpicture}
    \caption{A minimal graph when $n=4\times 3 = 12$ (the cliques are omitted)}
\end{figure}

For edge cases, when $k = 1$ we may simply take the complete graph $K_4$, which contains $6 = 4 \cdot 1^2 + 1 + 1$ edges. When $k = 2$, we may begin with the disjoint union of $K_2$ and $K_6$, then add 6 edges $u_iv_{i+j} (i = 0,1; j = 0,1,2)$. Finally, we remove a path $v_0v_1v_2v_3$. Now $u_i$ have degree 4 and all $v_j$ have degree 5. Since vertices with degree 4 are connected to each other, our graph satisfies the non-adjacent vertex condition, and the total number of edges will be $\binom{2}{2} + \binom{6}{2} + 6 - 3 = 19$.

We need to argue that $4k^2 + k + 1$ is the best we can do. The universal lower bound (Equation \ref{uniform-lower-bound}) gives us 
\begin{equation}
|E(G)| \geq 4k^2 + k - \frac{9}{16},
\end{equation}
and since $|E(G)|$ is an integer we have 
\begin{equation}
|E(G)| \geq 4k^2 + k.
\end{equation}
We are going to show $|E(G)| = 4k^2 + k$ cannot happen. First, if the minimum degree is not $2k-1$ or $2k$, we must have $|E(G)| > 4k^2 + k$. If the minimum degree is $2k-1$, the only way to achieve $|E(G)| = 4k^2 + k$ is we have $2k$ vertices of degree $2k+2$ and $2k$ vertices of degree $2k-1$. Since vertices of degree $2k-1$ need to be connected to each other, we see this scenario is impossible.

If the minimum degree is $2k$, then there can be at most $2k+1$ vertices with degree $2k$ since vertices of degree $2k$ must be connected to each other. If there are at most $2k-1$ vertices with degree $2k$, then we know $|E(G)| > 4k^2 + k$. If there are $2k$ vertices of degree $2k$, then we know the $2k$ vertices of degree $2k$ are connected to each other while there are $2k$ vertices of degree at least $2k+1$. For the $2k$ vertices of degree at least $2k+1$, each of them must be connected to at least 2 vertices of degree $2k$. For the $2k$ vertices of degree $2k$, as they are connected to each other, each vertex can only be connected to one vertex of degree at least $2k+1$. Thus, having $2k$ vertices of degree $2k$ is impossible. Similarly, we cannot have $2k+1$ vertices of degree $2k$ since those vertices must be connected to each other and cannot be connected to any other vertex.
\end{enumerate}

\end{proof}
\subsection{Number of Triangles}
In fact, for the non-adjacent vertex condition, we can still show when the number of vertices is odd, the graph minimizing the number of edges can also minimize the number of triangles.

\begin{theorem}\label{min-triangles}
Suppose $G$ is a graph with $n$ vertices, and $G$ satisfies for any pair of distinct vertices $x \neq y$, $xy \notin E(G)$, $d_G(x) + d_G(y) \geq n +1$. 
\begin{enumerate}
    \item If $n = 4k-1$ where $k \in \naturals$, $G$ contains at least $k(2k-1) = \frac{n^2-1}{8}$ triangles.
    
    \item If $n = 4k+1$ where $k \in \naturals$, $G$ contains at least $2k(k+1) = \frac{n^2+2n-3}{8}$ triangles.
\end{enumerate}
\end{theorem}

To prove this, we need the following fact:
\begin{fact}[\cite{lovaszsimon}]
Any graph with $n$ vertices and $\lfloor n^2 / 4 \rfloor + k$ edges contains at least $k\lfloor n/2 \rfloor$ triangles assuming that $k < n/2$.
\end{fact}

\begin{proof}[of Theorem \ref{min-triangles}]
As we have shown before, when $G$ contains $n = 4k-1$ vertices, $G$ must contain at least $4k^2 - k = \floor{\frac{(4k-1)^2}{4}} + k$ edges. Thus we know $G$ must contain at least 
\begin{equation}
    k\lfloor \frac{4k-1}{2} \rfloor = k(2k-1) = \frac{n^2-1}{8}
\end{equation}
triangles, and the construction that achieves $4k^2 - k$ edges above has $k(2k-1)$ triangles.

Similarly, if $G$ contains $n = 4k+1$ vertices, $G$ must contain at least 
\begin{math}
    4k^2 + 3k + 1 = \lfloor{\frac{(4k+1)^2}{4}}\rfloor + k+1
\end{math}
edges. Thus, $G$ must contain at least 
\begin{equation} 
(k+1)\lfloor \frac{4k+1}{2} \rfloor = 2k(k+1)
\end{equation}
triangles, and the construction we have above that achieves $4k^2 + 3k + 1$ edges has $2k(k+1)$  triangles.
\end{proof}

On the other hand, we want to show that if $G$ contains even number of vertices, the least number of triangles $G$ contains is closer to $\frac{n^2}{4}$:

\begin{theorem}
Suppose $G$ is a graph with $n = 2k$ vertices satisfying the non-adjacent vertex condition. Let $t_3^{\min} (n)$ denote the minimum possible number of triangles in $G$. Then 
\begin{equation}
     \lim_{n \rightarrow \infty}\frac{t_3^{\min}(n)}{n^2} = \frac{1}{4}.
\end{equation}
\end{theorem}
 
\begin{proof}
First, we are going to show that $\displaystyle \lim_{n \rightarrow \infty}\frac{t^{\min}_3(n)}{n^2} \leq \frac{1}{4}$ by giving explicit constructions that have at most $\frac{n^2}{4}$ triangles. The constructed graphs are $(k+1)$-regular graphs. 

If $n$ is divisible by 4, or in other words, $n = 4l$ for some $l$, then we may begin with the complete bipartite graph $K_{2l,2l}$ and add a matching inside each of the two parts with size $2l$. In the resulting graph, each vertex has degree $2l + 1 = k + 1$, and the number of triangles is simply 
\begin{equation} 
(2l) \cdot (2l) = 4l^2 = \frac{n^2}{4}.
\end{equation}

If $n$ is divisible by 2 but not 4, or in other words, $n = 4l+2$ for some $l$, we may begin with the complete bipartite graph $K_{2l,2l+2}$ and add a cycle inside the part with $2l+2$ vertices. Again, all vertices have degree $2l+2 = k+1$, and the number of triangles is simply 
\begin{equation} 
2l(2l+2) = \frac{n^2}{4} - 1.
\end{equation}

Next, we are going to show that $\displaystyle \lim_{n \rightarrow \infty}\frac{t^{\min}_3(n)}{n^2} \geq \frac{1}{4}$ by estimating the lower bound of triangles satisfying the non-adjacent vertex condition. First, we show to minimize the number of triangles, $G$ cannot contain a significant number of vertices with degree at most $k$. By the adjacent vertex condition, vertices with degree at most $k$ must form a clique. Thus, if $G$ contains $m$ vertices with degree at most $k$, 
\begin{equation} 
t_3(G) \geq \binom{m}{3} \geq \frac{(m-2)^3}{6}.
\end{equation}

As a result, if $\frac{(m-2)^3}{6} > \frac{n^2}{4}$, or in other words, $m > (\frac{3}{2})^{1/3}n^{2/3} + 2$, we must have $t_3(G) > \frac{n^2}{4}$.

On the other hand, if $m \leq (\frac{3}{2})^{1/3}n^{2/3} + 2$, we may show there will be a lot of triangles. First, if $m = 0$, then we know the number of edges will be at least $\frac{1}{2}(k+1)n = \frac{n^2}{4} + \frac{n}{2}$. Therefore, we know the graph contains at least $\frac{n}{2} \cdot (\frac{n}{2}-1) = \frac{n^2}{4} -\frac{n}{2}$ triangles.

If $m \neq 0$, we may assume that the minimum degree in the graph is $d \leq k$. If $v$ is a vertex of minimum degree, then the $n-d-1$ vertices that are neither $v$ itself nor its neighbors have degree at least $n+1-d$. Furthermore, we know that the 
\begin{math} 
n - (n-d-1) - m = d+1 - m
\end{math}
vertices left have degree at least $k+1$ by assumption, so we know the sum of degree is at least 
\begin{equation}
\label{sum-degree-bound-even}
\begin{aligned}
   & md + (n-d-1)(n+1-d) + (d+1-m) \left( \frac{n}{2}+1 \right)\\ 
   &= d^2 + \left(1 + m - \frac{3}{2}n \right)d + \left(n^2 + \frac{1}{2}n - m - \frac{1}{2}mn \right).
\end{aligned}
\end{equation}
If we disregard the condition that $d \leq k$, then when we view it as a function in $d$, we know it is minimized at $d = \frac{3n-2m-2}{4}$. However, since $m \leq (\frac{3}{2})^{1/3}n^{2/3} + 2$, we know that $\frac{3n-2m-2}{4} > \frac{n}{2}$ if $n$ is sufficient large ($n \geq 27$), and therefore we know that Equation \ref{sum-degree-bound-even}
is minimized at $d = k = \frac{n}{2}$. Therefore, we know that the number of edges the graph contains is at least $\frac{n^2 + 2n}{2} - m$. As $m \leq (\frac{3}{2})^{1/3}n^{2/3} + 2$, we know that 
\begin{equation} 
\frac{n^2 + 2n}{2} - m \geq \frac{1}{2}n^2 + n - (\frac{3}{2})^{1/3}n^{2/3} - 2.
\end{equation}
Therefore, the graph must contain at least 
\begin{equation} 
\frac{n^2}{4} + \frac{n}{2} - (\frac{3}{16})^{1/3}n^{2/3} - 1
\end{equation}
edges. Again, by Lov\'asz-Simonovits, we conclude 
\begin{equation} 
t_3^{\min}(n) \geq \frac{n}{2}(\frac{n}{2} - (\frac{3}{16})^{1/3}n^{2/3} - 1). 
\end{equation}
Clearly, 
\begin{equation}
\displaystyle \lim_{n \rightarrow \infty}\frac{n/2(n/2 - (\frac{3}{16})^{1/3}n^{2/3} - 1)}{n^2} = \frac{1}{4},
\end{equation}
and we may now conclude that 
\begin{equation} 
\displaystyle \lim_{n \rightarrow \infty} \frac{t_3^{\min}(n)}{n^2} = \frac{1}{4}.
\end{equation}

\end{proof}

\acknowledgements

The work presented here was done as part of the Budapest Semesters in Mathematics Undergraduate Research Program under the supervision of the first author. The work of the first author was supported by the National Research, Development and Innovation Office (Hungary) under Grant K132696.

\nocite{*}
\bibliographystyle{abbrvnat}
\bibliography{ref}
\label{sec:biblio}

\end{document}